\documentclass{amsart}
\usepackage{amssymb}
\usepackage{graphicx}
\usepackage[compress]{cite}

\newtheorem{thm}{Theorem}

\newtheorem{lem}{Lemma}

\newcommand{\A}{{\mathcal A}}

\newcommand{\U}{{\mathcal U}}
\newcommand{\es}{{\mathcal S}}

\newcommand{\D}{{\mathbb D}}

\author[M. Obradovi\'{c}]{Milutin Obradovi\'{c}}
\address{Department of Mathematics,
Faculty of Civil Engineering, University of Belgrade,
Bulevar Kralja Aleksandra 73, 11000, Belgrade, Serbia}
\email{obrad@grf.bg.ac.rs}

\author[N. Tuneski]{Nikola Tuneski}
\address{Department of Mathematics and Informatics, Faculty of Mechanical Engineering, Ss. Cyril and Methodius
University in Skopje, Karpo\v{s} II b.b., 1000 Skopje, Republic of North Macedonia.}
\email{nikola.tuneski@mf.edu.mk}

\subjclass[2020]{30C45, 30C50, 30C55}

\keywords{univalent functions, class $\U$, symmetric Toeplitz determinants, estimate.}

\title[Symmetric Toeplitz determinants of univalent functions]{Symmetric Toeplitz determinants of some classes of univalent functions}

\begin{document}

\begin{abstract}
In this paper we consider estimates of symmetric Toeplitz determinants $T_{q,n}(f)$ for the class $\U$ and for the general class $\es$ for certain values of $q$ and $n$ ($q,n=1,2,3\ldots$).
\end{abstract}

\maketitle

\section{Introduction and definitions}

\medskip

Let $\mathcal{A}$ denote the class of analytic functions in the open unit disc $\D=\{z:|z|<1\}$ with the form
\begin{equation}\label{e1}
f(z)=z+a_2z^2+a_3z^3+\cdots,
\end{equation}
i.e., satisfying $f(0)=f'(0)-1=0$. By $\es$, $\es\subset\A$, we denote the class of univalent functions in $\D$.

\medskip

For functions $f\in\A$ of form \eqref{e1} we define Hankel determinants by
\[
        H_{q,n}(f) = \left |
        \begin{array}{cccc}
        a_{n} & a_{n+1}& \ldots& a_{n+q-1}\\
        a_{n+1}&a_{n+2}& \ldots& a_{n+q}\\
        \vdots&\vdots&~&\vdots \\
        a_{n+q-1}& a_{n+q}&\ldots&a_{n+q-2}\\
        \end{array}
        \right |,
\]
where $q\geq 1$ and $n\geq 1$. Some examples of second order Hankel determinants are
\begin{equation}\label{e10}
\begin{split}
 H_{2,2}(f) &=  \left |
        \begin{array}{cc}
        a_2 & a_3\\
        a_3 & a_4\\
        \end{array}
        \right | = a_2a_4-a_{3}^2,\\[2mm]
 H_{2,3}(f) &=  \left |
        \begin{array}{cc}
        a_3 & a_4\\
        a_4 & a_5\\
        \end{array}
        \right | = a_3a_5-a_{4}^2.
\end{split}
\end{equation}

\medskip

The problem of finding upper bound of the Hankel determinant (preferebly sharp, i.e., best ones) is extensively studied in the past decade.
For the general class $\es$ of univalent functions few results concerning the Hankel determinant are known, and the best known for the second order case  is due to Hayman (\cite{hayman-68}), saying  that $|H_2(n)|\le An^{1/2}$, where $A$ is an absolute constant, and that this rate of growth is the best possible. 
Another one is \cite{Studia}, where it was proven that $|H_{2}(2)|\leq  1.3614356\ldots$ and $|H_{3}(1)|\leq  1.83056\ldots$, improvements of previous results from \cite{OT-S}.
There are much more results for the subclasses of $\es$ and some references  are \cite{ janteng-06,janteng-07, Kowalczyk-18, DTV-book}. In \cite{Kowalczyk-23}, the authors considered the cases of starlike, convex, strongly starlike and strongly convex functions and found the best possible results.

\medskip

Further, in their paper \cite{firoz} the authors considered the symmetric Toeplitz determinant $T_{q,n}(f)$ for functions $f\in\A$ of the form \eqref{e1} defined by
\[
T_{q,n}(f) = \left |
        \begin{array}{cccc}
        a_{n} & a_{n+1}& \ldots& a_{n+q-1}\\
        a_{n+1}&a_{n}& \ldots& a_{n+q-2}\\
        \vdots&\vdots&~&\vdots \\
        a_{n+q-1}& a_{n+q-2}&\ldots&a_{n}\\
        \end{array}
        \right |,
\]
where $q,n=1,2,3,\ldots$, and $a_1=1$. In particular it is easy to compute that
\begin{equation}\label{e3}
  \begin{split}
     T_{2,2}(f) &= a_2^2-a_3^2, \\
     T_{2,3}(f) &= a_3^2-a_4^2, \\
     T_{3,1}(f) &= 1-2a_2^2+2a_2^2a_3-a_3^2, \\
     T_{3,2}(f) &= (a_2-a_4)(a_2^2 - 2a_3^2 + a_2a_4), \\
     T_{3,3}(f) &= (a_3-a_5)(a_3^2 - 2a_4^2 + a_3a_5).
  \end{split}
\end{equation}

\medskip

In \cite{firoz}, the authors proved the next

\medskip

\noindent
{\bf Theorem A.} {\it If $f\in\es$ has the form \eqref{e1}, then
\[
|T_{2,2}(f)|\le13, \quad |T_{2,3}(f)|\le25, \quad |T_{3,1}(f)|\le24.
\]
All these results are sharp.}

\medskip

In the same paper similar problems were considered for different subclasses of $\es$, such as the classes of starlike and convex functions, and other.

\medskip

In this paper we will study the class $\U$, $\U\subset\es$, defined by the condition
\[\left| \left(\frac{z}{f(z)}\right)^2 f'(z) -1\right|<1 \quad (z\in\D).\]
It is known that this class is not a subset of the class of starlike functions, nor vice-versa, which is rare property and makes it attractive. More about the class $\U$ can be found in \cite{jammu, mari}.

\medskip

Previously, for our work we will need the following lemmas.

\begin{lem}\label{lem-1}
  Let  $f(z)=z+a_2z^2+\cdots$. Then
\begin{equation}\label{star}
  f\in\U \quad \Leftrightarrow \quad \frac{z}{f(z)} = 1-a_2z-z \omega(z),
\end{equation}  
where $\omega(0)=0$, and $|\omega(z)|< 1$, $|\omega'(z)|\le 1$ for all $z\in\D$. From \eqref{star}, for $\omega(z)=c_1z+c_2z^2+\cdots$, follows  
\begin{equation}\label{eq-3}
 a_3 = a_2^2+c_1,\quad a_4 = c_2+2a_2c_1+a_2^3,\quad a_{5}=c_3+2a_2c_2+c_{1}^{2}+3a_2^2 c_{1}+a_{2}^{4},
\end{equation}
where
\begin{equation}\label{eq-5}
|c_{1}|\leq 1,\quad |c_2|\leq\frac{1}{2}\left(1-|c_1|^2\right),\quad   |c_{3}|\leq\frac{1}{3}\left(1-|c_{1}|^{2}-\frac{4|c_{2}|^{2}}{1+|c_{1}|}\right).
\end{equation}
\end{lem}

\begin{proof} If $f(z)=z+a_2z^2+\cdots\in \U$, then expression (4) from \cite{mari} leads to $\frac{z}{f(z)} = 1-a_2z-z \omega(z)$. Vice versa, if $\frac{z}{f(z)} = 1-a_2z-z \omega(z)$, then it can be checked directly that $f$ is in $\U$. Expressions \eqref{eq-3} and \eqref{eq-5} follow from (6) and (7) in \cite{mari}, respectively.
\end{proof}

\begin{lem}[\cite{mari, tartu}]\label{lem-2}
Let $f\in\U$. Then:
\begin{itemize}
  \item[(a)] $|H_{2,2}(f)|\le1$;\vspace{1mm}
  \item[(b)] $|H_{2,3}(f)|\le 1.4946575\ldots$;\vspace{1mm}
  \item[(c)] If $a_2=0$, $|H_{2,3}(f)|\le 1$.
\end{itemize}
Estimates (a) and (c) are sharp.
\end{lem}

\begin{lem}[\cite{tartu, pale}]\label{lem-3}
Let $f\in\es$. Then:
\begin{itemize}
  \item[(a)] $|H_{2,2}(f)|\le1.3614\ldots$;\vspace{1mm}
  \item[(b)] $|H_{2,3}(f)|\le 4.89869\ldots$;\vspace{1mm}
  \item[(c)] If $a_2=0$, $|H_{2,3}(f)|\le 2.02757\ldots$.
\end{itemize}
\end{lem}

\medskip

\section{Main results}

\begin{thm}\label{th-1}
Let $f\in\U$ be of the form \eqref{e1}. Then
\begin{itemize}
  \item[($i$)] $|T_{2,2}(f)|\le 13$;\vspace{1mm}
  \item[($ii$)] $|T_{2,3}(f)|\le 25$;\vspace{1mm}
  \item[($iii$)] $|T_{3,1}(f)|\le 24$;\vspace{1mm}
  \item[($iv$)] $|T_{3,2}(f)|\le 84$;\vspace{1mm}
  \item[($v$)] $|T_{3,3}(f)|\le 211.8771\ldots$.
\end{itemize}
The inequalities (i)-(iv) are sharp.
\end{thm}

\begin{proof}
The estimates (i) and (ii) easily follow from
\[ |T_{2,2}(f)|\le |a_2|^2+|a_3|^2 \quad\mbox{and}\quad |T_{2,3}(f)|\le |a_3|^2+|a_4|^2, \]
and $|a_2|\le 2$, $|a_3|\le 3$, $|a_4|\le 4$, for the class $\U$.
\medskip
\begin{itemize}
  \item[(iii)] From Lemma \ref{lem-1}, after some calculations we receive
  \[ T_{3,1}(f) = 1-2a_2^2+a_2^4-c_1^2 = 1-2a_2^2+(a_2^2-c_1)(a_2^2+c_1), \]
  and from here
  \[
  \begin{split}
  |T_{3,1}(f)| &\le 1+2|a_2|^2 + \left||a_2|^2+|c_1|\right|\cdot \left|a_2^2+c_1\right|\\
  &\le 1+2\cdot 4+(4+1)\cdot3 = 24,
  \end{split}
  \]
  since $|a_3|=|a_2^2+c_1|\le3$, $|a_2|\le2$, $|c_1|\le1$ (see Lemma \ref{lem-1}).
  \medskip
  \item[(iv)] From \eqref{e3} we have
  \[
  \begin{split}
   T_{3,2}(f) &= (a_2-a_4)(a_2^2 - 2a_3^2 + a_2a_4)\\
   & = (a_2-a_4)\left[(a_2^2 - a_3^2) + (a_2a_4-a_3^2)\right],
     \end{split}
     \]
  and from Lemma \ref{lem-2}(a),
  \[ |T_{3,2}(f)| \le (|a_2|+|a_4|)\left[|a_2|^2 + |a_3|^2) + |H_{2,2}(f)|\right] \le 6\cdot 14 = 84. \]
  \item[(v)] Similarly, using \eqref{e3} and Lemma \ref{lem-2}(b), we obtain
  \[
  \begin{split}
      |T_{3,3}(f)| &\le (|a_3|+|a_5|)\left(|a_3|^2 + |a_4|^2 + |H_{2,3}(f)|\right)\\
      & \le 8\cdot(25+1.4846575\ldots) = 211.4846575\ldots.
  \end{split}
  \]
\end{itemize}
\medskip

The estimates (i)-(iv) are sharp as the function
\[ f_1(z) = \frac{z}{(1-iz)^2} = z+2iz^2-3z^3-4iz^4+5z^5+\cdots \]
shows. At same time $|T_{3,3}(f_1)| = 208$.
\end{proof}

\begin{thm}
Let $f\in\U$ be of the form \eqref{e1} with $a_2=0$. Then
\begin{itemize}
  \item[($i$)] $|T_{2,2}(f)|\le 1$;\vspace{1mm}
  \item[($ii$)] $|T_{2,3}(f)|\le 1$;\vspace{1mm}
  \item[($iii$)] $|T_{3,1}(f)|\le 2$;\vspace{1mm}
  \item[($iv$)] $|T_{3,2}(f)|\le \frac{3}{16}$;\vspace{1mm}
  \item[($v$)] $|T_{3,3}(f)|\le \frac92$.
\end{itemize}
The inequalities (i)-(iv) are sharp.
\end{thm}

\begin{proof}
From Lemma \ref{lem-1} and using $a_2=0$, we have
\[
\begin{split}
 |a_3| &=|c_1|\le 1,\\
|a_4| &=|c_2|\le \frac12 (1-|c_1|^2) \le \frac12,\\
 |a_5| &= |c_3 + c_1^2| \le |c_3|+|c_1|^2 \le \frac13\left( 1-|c_1|^2 -\frac{4|c_2|^2}{1+|c_1|}\right) +|c_1|^2 \\
&\le \frac13 +\frac23|c_1|^2 \le 1.
\end{split}
\]
So, by \eqref{e3} we have
\begin{itemize}
  \item[(i)] $|T_{2,2}(f)| = |-a_3^2| \le 1$;
  \item[(ii)] $|T_{2,3}(f)| = |c_1^2 - c_2^2| \le  |c_1|^2 + |c_2|^2 \le |c_1|^2 +\frac14 \left( 1-|c_1|^2 \right)^2  = \frac14 + \frac12|c_1|^2 +\frac14 |c_1|^4 \le 1$;
  \item[(iii)] $|T_{3,1}(f)| = |1-a_3^2| \le 1+|c_1|^2 \le 2$;
  \item[(iv)]  $|T_{3,2}(f)| = 2|a_3|^2|a_4| \le 2|c_1|^2|c_2|\le  2|c_1|^2\cdot \frac12(1-|c_1|^2)\le    \frac{3}{16}$;
  \item[(v)] $|T_{3,3}(f)| = (|a_3|+|a_5|)\left( |a_3|^2+|a_4|^2+|H_{2.3}(f)| \right) \le 2\cdot\left( 1+\frac14+1 \right) =\frac92$.
  \end{itemize}

  \medskip

  The estimates (i) and (ii) are sharp due to the function $f_2(z)=\frac{z}{1-z^2} = z+z^3+z^5+\cdots$, while (iii) is sharp due to $f_3(z)=\frac{z}{1-iz^2} = z+iz^3-z^5\cdots$.
  
  In the estimate (iv), equality is attained for $|c_1|^2=\frac12$, i.e., for $|c_1|=\frac{1}{\sqrt2}$. The result is sharp with extremal function $f_4$ such that
  \[ \frac{z}{f_4(z)} = 1-z \int_0^z \frac{1/\sqrt2 + t}{1 + 1/\sqrt2t}dt, \]
  well defined because by equating coefficients we receive $a_2=0$ and the function $\omega_1(z) = \int_0^z \frac{1/\sqrt2 + t}{1 + 1/\sqrt2t}dt$ has the properties $\omega_1(0)=0$, and $|\omega_1(z)|< 1$, $|\omega_1'(z)|\le 1$ for all $z\in\D$. 
\end{proof}

\begin{thm}\label{th-3}
If $f\in\es$ has the form \eqref{e1}, then
\begin{itemize}
  \item[($i$)] $|T_{3,2}(f)|\le 86.1684\ldots$;\vspace{1mm}
  \item[($ii$)] $|T_{2,3}(f)|\le 239.1895\ldots$.
\end{itemize}
\end{thm}

\begin{proof}$ $
\begin{itemize}
  \item[(i)] Similarly as in the proof of Theorem \ref{th-1} we have
 \[
 \begin{split}
 |T_{3,2}(f)| &\le  \left(|a_2|+|a_4|\right)\left(|a_2|^2 + |a_3|^2 + |H_{2,2}(f)|\right) \\
 & \le 6\cdot(13+1.3614\ldots) = 86.1684\ldots,
 \end{split}
 \]
 where we used Lemma \ref{lem-3}(a).
    \item[(ii)] Also,
 \[
 \begin{split}
 |T_{3,3}(f)| &\le  \left(|a_3|+|a_5|\right)\left(|a_3|^2 + |a_4|^2 + |H_{2,3}(f)|\right) \\
 & \le 8\cdot(25+4.89869\ldots) = 239.1895\ldots,
 \end{split}
 \]
 where we used Lemma \ref{lem-3}(b).
\end{itemize}
\end{proof}

\begin{thm}
If $f\in\es$ has the form \eqref{e1} with $a_2=0$, then
\begin{itemize}
  \item[($i$)] $|T_{3,2}(f)|\le \frac43$;\vspace{1mm}
  \item[($ii$)] $|T_{2,3}(f)|\le 7.3883\ldots$.
\end{itemize}
\end{thm}

\begin{proof}
Since $a_2=0$, then by \cite{OTT, tbilisi} we have
$|a_3| \le 1$, $|a_4|\le\frac23$, $|a_5|\le \frac34+\frac{1}{\sqrt7}=1.12796\ldots$, $|H_{2,2}(f)|\le1 $, and $|H_{2,3}(f)|\le 2.02757\ldots $ (by Lemma \ref{lem-3}(c). We receive the estimates by applying the same method as in Theorem \ref{th-3}.
\end{proof}

\end{document}